\documentclass{amsart}

\usepackage{amssymb}
\usepackage{amsmath}

\newcommand{\ds}{\displaystyle}

\newtheorem{theorem}{Theorem}
\newtheorem{lem}{Lemma}
\newtheorem{defn}{Definition}
\newtheorem{prop}{Proposition}
\newtheorem{coro}{Corollary}

\begin{document}

\title[On The $L^{2}-$Solutions of S.F.P.D.Es]{On The $
L^{2}-$Solutions of Stochastic Fractional Partial Differential
Equations; Existence, Uniqueness and Equivalence of Solutions}

\author{Latifa DEBBI}

\begin{abstract}
The aim of this work is to prove existence and uniqueness of
$L^{2}-$solutions of stochastic fractional partial differential
equations in one spatial dimension. We prove also the equivalence
between several notions of $L^{2}-$solutions. The Fourier transform
is used to give meaning to SFPDEs. This method is valid also when
the diffusion coefficient is random.
\end{abstract}



\address{Institut Elie Cartan, Nancy 1 B.P 239, 54506
Vandoeuvre-Les-Nancy cedex, France \& Department of Mathematics,
Faculty of Sciences, University Ferhat Abbas, El-Maabouda S\'{e}tif
19000, Algeria.}

\email{ldebbi@yahoo.fr, debbi@iecn.u-nancy.fr}

\maketitle

\section{Introduction}
Fractional calculus and stochastic analysis are connected concepts
thanks to the selfsimilarity property. In recent years,
mathematicians as well as physicians draw more attention to the use
of the two topics simultaneously to model complex phenomena. Several
definitions of fractional differential operators have been
introduced based on probabilistic concepts, see for short list e.g.
\cite{F2, JacobI, JacobII, UZ1}. Moreover, several phenomena, which
are described to be anomalous, are modeled using fractional calculus
and/or stochastic analysis, see e.g \cite{ARAG1, AHL1, BensWheMee,
Bilerandal, BrzezniakDebbi1, Caffarelli-2009,
Caffarelli-Vasseur2010, D1, D2, GioRom, GM1, UZ1, W2, ZabPFrac,
Zaslavsky1, Zaslavsky2, ZasABD} and the references therein. A
phenomenon is described as anomalous if it is not covered by the
Gaussian Markovian case. The anomaly is characterized by the long
range dependence (LRD) effect and/or by the coexistence of the
diffusive and the ballistic modes. One way to model the anomaly is
to consider stochastic partial differential equations (shortly
SPDEs) perturbed by non Gaussian noises, such as the L\'evy or/and
non Markovian noises, such as the fractional Brownian motion, see
e.g.\cite{Hu1, M1, NuaRasc, NuaVuil} and others. The main difficulty
in the study of SPDEs perturbed by non Markovian processes is due to
the lack of a standard stochastic integral theory. To encounter this
difficulty, SPDEs driven by fractional operator are used. Here the
anomaly is presented via the Green function of the fractional
operator, see e.g. ~\cite{ARAG1, BrzezniakDebbi1, DM1, DalaSanz1,
DD1, W2}. In these later works, authors are interested in the
existence, uniqueness and the regularity of the solutions of
different kinds of stochastic partial differential equations (SPEDs)
driven by fractional operators. In ~\cite{ARAG1}, a linear SPED
driven by the composition of the inverses of Riesz and Bessel
potentials and perturbed by a space-time white noise is studied.
In~\cite{DM1}, the authors proved the existence and the uniqueness
of the solution of an hyperbolic multidimensional SPDE driven by a
power Laplacian and perturbed by a colored noise; white in time and
homogeneous in space. The regularity of the solution is obtained in
\cite{DalaSanz1}.  In \cite{DD1}, the authors considered high order
stochastic fractional partial differential equations with entire
derivatives and perturbed by space-time white noise. The
non-Lipschitz case is treated in \cite{BrzezniakDebbi1, W2}. In
particular, the stochastic Burgers equation driven by fractional
power of the Laplacian and perturbed by a cylindrical white noise
respectively by a stable noise is studied. One of the main results,
was the precision of the tree interaction between the dissipation,
given by the fractional operator, the steepening, given by the
nonlinear term and the regularity of the random noise.

The aim of this work is to prove existence and uniqueness of $
L^{2}$-solutions of the SFPDEs introduced in \cite{DD1}. The
$L^{2}$-solution obtained in this paper coincides with the solution
obtained in \cite{DD1} under some special class of Lipschitz
conditions. We present three different notions of $
L^{2}$-solutions, mild, weak of first kind and weak of second kind.
Moreover, we prove that these solutions are equivalent, for the
literature on equivalence solutions, see e.g. \cite{Chojnowska,
Leon} for the evolutive SPDEs, \cite{SanzVuirmo} for the
quasi-evolutive case and \cite{GongyPrint, GN} for the Walsh's
approach. The result generalizes the equivalence obtained in
\cite{GongyPrint, G}. A special section is devoted to give meaning
to SFPDEs using Fourier transform. This method is relevant when the
diffusion coefficient is random and depends only on the spatial and
the temporal variables but not on the solution. The study of the
SFPDEs reduce to SDEs driven by martingales. The Fourier transform
of the solution is a generalization of the Ornstein-Uhlenbeck
process. We prove that the solution of the equation without
derivatives of entire order given via the Fourier technique is
equivalent to the mild solution.

The paper is organized as follows. In section 2, we prove existence
and uniqueness of $L^2$-mild solutions. In section 3, we prove the
equivalence of mild and  weak solutions. In section 4, we apply the
Fourier technique to define a solution for a special case of the
equation studied. This notion of solution is equivalent to mild
solution and to weak solutions.

We are interested in the following Cauchy problem:
\begin{equation}\label{Eq.Principal}
\Bigg\{
\begin{array}{lr}
\ds \frac{\partial u}{\partial t}(t,x) =
{}_{x}D_{\delta}^{\alpha}\ds u(t,x)+
\sum_{k=0}^{m}\frac{\partial^{k}h_{k}}{\partial x^{k}}(t, x, u(t, x))+ f(t,x,u(t, x))\ds \frac{\partial^{2}W}{\partial t\partial x}(t,x), \\
t>0, \quad x\in
\mathbb{R}, \\
u(0, x)= u^{0}(x), \\
\end{array}
\end{equation}
where $\alpha \in \mathbb{R}_{+}\backslash \mathbb{N} $, $ m
\in\mathbb{N}$, such that $ 1 \leq m \leq [\alpha] $, where
$[\alpha]$ is the integer part of $ \alpha $ and $
{}_{x}D_{\delta}^{\alpha} $ is the fractional differential operator
with respect to the spatial variable, to be defined below. We
suppose that the functions $f, g, h_{k}:
[0,+\infty)\times\mathbb{R}\times\mathbb{R}\rightarrow\mathbb{R}$
satisfy Lipschitz and growth conditions:

for all $ T > 0 $, there exist a constant $K_{T}> 0 $ and functions
$ a_{k} \in L^{2}(\mathbb{R}),\; a_{k}\geq 0,\;  k= 0, 1, ..., m+1 $
such that for all $ t\in [0,T]$ and for all $ x, y, z \in
\mathbb{R}$
\begin{eqnarray}\label{Lipschitz Cond.}
\Big(\left|h_{k}(t,x,y)-h_{k}(t,x,z)\right|+
\left|f(t,x,y)-f(t,x,z)\right|\Big)
\leq K_{T} \left|y-z\right|,\nonumber\\
\left|h_{k}(t,x,z)\right|\leq K_{T}(a_{k}(x) + \left|z\right|),
\quad \left|f(t,x,z)\right| \leq K_{T}(a_{m+1}(x)+\left|z\right|).
\end{eqnarray}
It is clear that when $ a_{k}\in L^{2}(\mathbb{R})\cap
L_{\infty}(\mathbb{R})$, we find the Lipschitz conditions in
\cite{DD1}. Let $(\Omega,\mathcal{F}, P)$ be a complete probability
space and let $W = \{W(t,x), t\geq 0, x\in\mathbb{R}\}$ be a
centered Gaussian field defined on $(\Omega,\mathcal{F}, P)$ with
covariance function given by
$$K((t,x),(s,y))=\frac{1}{4}(sgn(x) + sgn(y))^{2}(t\wedge s)(|x| \wedge
|y|),$$ where "$ sgn $" denoted the sign function. $W$ is in fact
composed of two independent Brownian sheets, one in the positive
direction of the spatial variable and the other one in the negative
direction. Let $(\mathcal{F}_t,t\geq 0)$ be an increasing and
right-continuous filtration generated by $ W $.  The initial
condition $u^{0}$ is supposed to be a $\mathcal{F}_{0}-$measurable $
L^{2}(\mathbb{R})-$valued function. We suppose that $ \alpha >1$ and
$ p \geq 1 $.


\begin{defn}\label{D.Frac.Diff.Func.}
Let $\alpha \in \mathbb{R}_{+}$. The $\alpha-$fractional derivative
operator is defined for all $ f\in  \{ g \in L^{2}(\mathbb{R}) /
|\lambda|^{\alpha}\hat{g}(\lambda) \in L^{2}\}, $  by
\begin{equation}\label{Def.Frac.Der}
D_{\delta}^{\alpha} f = \mathcal{F}^{-1}(
{}_{\delta}\psi_{\alpha}(.)\hat{f}),
\end{equation}
where $ \left| \delta \right|\leq \min \{ \alpha - [\alpha]_{2}, 2+
[\alpha]_{2}- \alpha\} $,  $ [\alpha]_{2} $ is the largest even
integer less than $\alpha $ (even part of $\alpha$) and $\delta = 0
$ when $ \alpha \in 2\mathbb{N}+1$,
\begin{equation}\label{Def.Psi}
{}_{\delta}\psi_{\alpha}(\lambda)=
-\left|\lambda\right|^{\alpha}e^{-i \delta\frac{\pi}{2} sgn
\lambda},
\end{equation}
and $\mathcal{F}^{-1}$ is the inverse Fourier transform on $
\mathbb{R}$ and $\hat{f}$ is the Fourier transform of $f$.
\end{defn}
The Fourier transform and its inverse are given by
\begin{equation}
\begin{array}{rl}
\mathcal{F}\{f(x);\lambda \} = \hat{f}(\lambda )=\int_{-\infty
}^{+\infty }\exp (ix\lambda )f(x)dx, \\
\mathcal{F}^{-1}\{f(\lambda);x \}=\frac{1}{2\pi}\int_{-\infty
}^{+\infty }\exp (-ix\lambda )f(\lambda)d\lambda.
\end{array}
\end{equation}
The operator $ D_{\delta}^{\alpha} $ is the infinitesimal generator
of an analytic semigroup of convolution given by the Green function
$ {}_{\delta}G_{\alpha}(t, x)=
\mathcal{F}^{-1}\{\exp[{}_{\delta}\psi_{\alpha}(\lambda)t]; x\}$.
Hence it is closed densely defined operator. The function
${}_{\delta}G_{\alpha }(t, x)$ is real but it is not symmetric
relatively to $x$, when $ \delta \neq 0$. Further, it is not
everywhere positive when $\alpha >2$. However, $\int_{-\infty
}^{+\infty }{}_{\delta}G_{\alpha }(t, x)dx = 1$. The explicit form
of $ {}_{\delta}G_{\alpha}(t, .) $ is known only for $ \alpha \in\{
\frac{1}{2}, 1, 2 \}$. Moreover, $ {}_{\delta}G_{\alpha}(t, .) $ has
a polynomial decrease when $ \alpha \not \in \mathbb{N}$. For more
details on this operator and the properties of $
{}_{\delta}G_{\alpha}(t, .) $ see \cite{D1, D2, DD1}. In the
following Lemma, we give some of the properties of the function $
{}_{\delta}G_{\alpha}(., .) $ that we need in this context.

\begin{lem}\label{properties of G}

\noindent

(i) ${}_{\delta}G_{\alpha }(t, x)$ satisfies the semi-group
property, or the Chapman Kolmogorov equation,  i.e.  for $ 0 < s < t
$
\begin{equation*}
{}_{\delta}G_{\alpha }(t+s, x)=\int_{-\infty }^{+\infty
}{}_{\delta}G_{\alpha }(t, \xi ){}_{\delta}G_{\alpha }(s, x-\xi
)d\xi,
\end{equation*}

(ii)  For $ 0 < \alpha \leq 2$, the function $
{}_{\delta}G_{\alpha}(t,.) $ is the density of a L\'{e}vy stable
process in time $ t $,

(iii) For fixed $ t$, $ {}_{\delta}G_{\alpha }(t, .)\in S^{\infty
}=\{f\in C^{\infty }$ and $ D_{\delta'}^{\beta }f $ is bounded and
tends to zero when $ \left\vert x\right\vert $ tends to $\infty $
$,\forall \beta \in \mathbb{R}_{+}, \left| \delta' \right|\leq \min
\{ \beta - [\beta]_{2}, 2+ [\beta]_{2}- \beta\} \mbox{ and } \delta'
= 0 \mbox{ when}\;  \beta \in 2\mathbb{N}+1 \}$,

(iv) $\frac{\partial^{l} }{\partial
x^{l}}{}_{\delta}G_{\alpha}(t,x)=t^{-{\frac{l+1}{\alpha}}}
\frac{\partial^{l} {}_{\delta}G_{\alpha}}{\partial
\xi^{l}}(1,\xi)|_{\xi=t^{-{\frac{1}{\alpha}}}x}, $ for all $ l\geq 0
$ (when l= 0, it is called the scaling property),

(v) $\frac{\partial^{l}}{\partial x^l}{}_{\delta}G_{\alpha}(1, x)=
\frac{1}{\pi}\sum_{j=1}^{n}|x|^{-\alpha
j-(l+1)}\frac{(-1)^{j+l}}{j!}\Gamma(\alpha j+l+1)\sin
j\frac{(\alpha+\delta)}{2}\pi + O(|x|^{-\alpha (n+1)-(l+1)}),$ when
$ |x| $ is large.
\end{lem}

\begin{coro}\label{Corollary}
Let $\alpha > 1$. For any fixed $ k \in \mathbb{N} $, for $ \gamma >
\frac{1}{\alpha+k+1}$,
$$
\left|{}_{\delta}G_{\alpha}^{(k)}(t,.)\right|_{\gamma}= K_{\alpha,
\gamma, k}t^{\frac{1-(k+1)\gamma}{\alpha\gamma}}.
$$
\end{coro}

Let us also give the following stochastic Fubini's Theorem for
Brownian sheet with respect to a deterministic non negative measure.

\begin{lem}(Stochastic Fubini's Theorem)
Let $( \mathcal{X}, \mathcal{B}(\mathcal{X}), \mu )$ be a measure
space and let $ f: \Omega\times \mathbb{R}_{+}\times
\mathbb{R}\times \mathcal{X} \rightarrow \mathbb{R}$ such that,
$\forall t > 0 $, the function $f$ is  $ \mathcal{F}_{t}\times
\mathcal{B}([0, t])\times\mathcal{B}(\mathbb{R})\times
\mathcal{B}(\mathcal{X})$-measurable  and
$$
\int_{0}^{t}\int_{\mathbb{R}}\mathbb{E}\Big(\int_{\mathcal{X}}f(s,
y, x)\mu(dx)\Big)^{2}dyds<\infty.
$$
Then the integrals $$\label{Eq. Fubini's lemma }
\int_{0}^{t}\int_{\mathbb{R}}\int_{\mathcal{X}}f(s, y,
x)\mu(dx)W(dyds), \qquad
\int_{\mathcal{X}}\int_{0}^{t}\int_{\mathbb{R}}f(s, y,
x)W(dyds)\mu(dx)
$$
are well defined and are $ \mathbf{P}- a. s. $ equal.
\end{lem}

Let $L^{p}(\Omega, \mathcal{F}_{0}, L^{2})= \{X: \Omega \rightarrow
L^{2}(\mathbb{R}), X \; \mbox{is }\;
\mathcal{F}_{0}-\mbox{measurable and such that }
\mathbb{E}|X|_{2}^{p}<\infty \}$. The scalar product in
$L^{2}(\mathbb{R})$ is denoted by $ \langle . , . \rangle $ and the
norm by $|.|_2$. We note also that the value of the constants in
this paper may change from line to line and some of the standing
parameters are not always indicated. In particular, the dependence
on $ T $.

\section{Existence and Uniqueness of Solution}

It is known that the equation (\ref{Eq.Principal}) has no rigorous
meaning. In the following definition, we give the notion of $
L^{2}-$mild solution.
\begin{defn}\label{Def.mild.Solu.}
A $ L^{2}-$valued $ \mathcal{F}_t-$adapted stochastic process $u =\{
u(t,.), t\in [0,T]\}$ is said to be a mild solution of the SFPDE in
(\ref{Eq.Principal}) on the interval $ [0,T] $, with initial
condition $ u^0 $ if it satisfies the following integral equation
for all $ t \in [0, T] $,
\begin{equation}\label{Eq.Mild.Solu}
\begin{array}{rl}
u(t,.)=&\ds\int_{\mathbb{R}}\ds {}_{\delta}G_{\alpha}(t,.-y)u^{0}(y)dy  \\
& + \sum_{k=0}^{m}(-1)^{k} \ds \int_{0}^{t}\ds\int_{\mathbb{R}} \ds
\ds h_{k}(s,y,u(s,y))
\frac{\partial^{k}}{\partial z^{k}}{}_{\delta}G_{\alpha}(t-s, z)|_{.-y}\ dyds \\
 & + \ds
\int_{0}^{t}\int_{\mathbb{R}}f(s,y,u(s,y)){}_{\delta}G_{\alpha}(t-s,.-y)W(dyds).
\end{array}
\end{equation}
The equality in (\ref{Eq.Mild.Solu}) is taken in $ L^{2}(\mathbb{R})
$.
\end{defn}

\begin{theorem}\label{Exist.mild.solu.}
Let $ \alpha >1$ and let $ u^{0} \in L^{p}(\Omega, \mathcal{F}_{0},
L^{2})$, where $ p\geq 1$. Then under conditions (\ref{Lipschitz
Cond.}), the equation (\ref{Eq.Principal}) admits a unique mild
solution which satisfies the inequality
\begin{equation}\label{UnifBoundIneqSolu}
\sup_{[0, T]}\mathbb{E}|u(s)|_{2}^p<\infty.
\end{equation}
The uniqueness is taken with respect to the norm in the left hand
side of (\ref{UnifBoundIneqSolu}).
\end{theorem}

\begin{proof}

\noindent

Let $ \mathcal{H} $ be a Banach space of $ L^{2}-$valued $
\mathcal{F}_{t}-$adapted processes endowed by the norm
$$ |u|_{\mathcal{H}}^{p}:= \int_{0}^{T}e^{-\lambda t}\mathbb{E}|u(t)|_{2}^{p}dt <
\infty,$$  where $ \lambda > 0 $ will be determined later. Let $
\mathcal{H}^*$ denote the subspace of the processes of $\mathcal{H}$
satisfying the assumption (\ref{UnifBoundIneqSolu}). We define on $
\mathcal{H}$ the operator $ \mathcal{A}$ by
\begin{equation}
\mathcal{A}u= \displaystyle \sum_{k=0}^{m+2}\mathcal{A}_{k}u,
\end{equation}
where
\begin{equation}\label{les Ak}
\begin{array}{lr}
\mathcal{A}_{0}u(t) \quad  = \int_{\mathbb{R}}\ds
{}_{\delta}G_{\alpha}(t,.-y)u^{0}(y)dy,
\nonumber\\
\displaystyle\mathcal{A}_{k+1}u(t) = \displaystyle (-1)^{k}\ds
\int_{0}^{t}\ds \ds \int_{\mathbb{R}} \ds \ds
h_{k}(s,y,u(s,y))\frac{\partial^{k} {}_{\delta}G_{\alpha}}{\partial
z^{k}}(t-s,z)|_{.-y}\ dyds, \quad 0 \leq k
\leq m \nonumber\\
\displaystyle\mathcal{A}_{m+2}u(t)=  \displaystyle
\int_{0}^{t}\int_{\mathbb{R}}f(s,y,u(s,y)){}_{\delta}G_{\alpha}(t-s,.-y)W(dyds).\\
\end{array}
\end{equation}
From the sequel it is easy to deduce that the operator $
\mathcal{A}$ takes  $\mathcal{H}$ to the space of $
\mathcal{F}_{t}-$adapted processes $\{u(t), t\geq 0\}$ such that for
almost all $t$, we have $ u(t) \in L^2(\mathbb{R}) \; a.s $. We
prove that $\mathcal{H}^*$ is an invariant subspace for the operator
$ \mathcal{A}$. The restriction of $ \mathcal{A}$ on $\mathcal{H}^*$
will be denoted by $ \mathcal{A}$ too. In fact, let $ u\in
\mathcal{H}^*$. It is easy to see that all the terms in the right
hand side of (\ref{les Ak}), are $ \mathcal{F}_{t}-$adapted
processes when they exist. Further,

$\mathcal{A}_{0}u(t) \in \mathcal{H}^* $ thanks to the assumption $
u^{0} \in L^{p}(\Omega, \mathcal{F}_{0}, L^{2})$ and to the
inequality
\begin{equation}\label{IneqA0u(t)}
|\mathcal{A}_{0}u(t)|_{2} = |G_{\alpha}(t, .)* u^{0}|_{2}\leq
|G_{\alpha}(t, .)|_{1}|u^{0}|_{2} \leq K |u^{0}|_{2} \quad a.s.
\end{equation}
For $ \mathcal{A}_{k}u(t)$, $ k = 0, 1, ..., m $, apply generalized
Minkowsky's inequality, Young's inequality, corollary
\ref{Corollary} and conditions (\ref{Lipschitz Cond.}) to get
\begin{eqnarray}\label{Est.give expect Ak+1Intermd}
\mathbb{E}|\mathcal{A}_{k+1}u(t)|_{2}^p& =& \mathbb{E}\Big(\ds
\int_{\mathbb{R}}\ds \Big|\int_{0}^{t}\ds \ds h_{k}(s,u(s))*
{}_{\delta}G_{\alpha}^{(k)}(t-s, x-.)ds \Big|^{2}dx\Big)^{\frac{p}{2}}\nonumber\\
& \leq & \ds\mathbb{E}\Big( \int_{0}^{t}\{ \ds\int_{\mathbb{R}}
|h_{k}(s,u(s))*
{}_{\delta}G_{\alpha}^{(k)}(t-s,x-.)|^{2}dx\}^{\frac{1}{2}}ds\nonumber\Big)^p\\
& \leq &\ds \mathbb{E}\Big(\int_{0}^{t} |h_{k}(s,u(s))|_{2}
|{}_{\delta}G_{\alpha}^{(k)}(t-s,.)|_{1}ds\Big)^p\nonumber\\
& \leq& K\ds \mathbb{E}\Big(\int_{0}^{t}
(t-s)^{-\frac{k}{\alpha}}\big(|a_{k}|_{2}+
|u(s)|_{2}\big)ds\Big)^p\nonumber\\
& \leq& K \mathbb{E}\big( 1+ \ds \int_{0}^{t}
(t-s)^{-\frac{k}{\alpha}}
|u(s)|_{2}^pds\big) \quad \nonumber\\
& \leq& K \big( 1+ \ds \int_{0}^{t} (t-s)^{-\frac{k}{\alpha}}
\mathbb{E}|u(s)|_{2}^pds\big). \nonumber\\
\end{eqnarray}
This proves, on one hand that
$\sup_{[0,T]}\mathbb{E}|\mathcal{A}_{k+1}u(t)|_{2}^{p}< \infty$ i.e.
$ \{\mathcal{A}_{k+1}u(t), \; t\geq 0\}  $ is an $L^2-$valued
process and satisfies (\ref{UnifBoundIneqSolu}). On the other hand,
thanks to Fubini's Theorem, we get
\begin{eqnarray}\label{Est.give the norm Ak+1}
\quad \int_{0}^{T}e^{-\lambda
t}\mathbb{E}|\mathcal{A}_{k+1}u(t)|_{2}^{p}dt &\leq & K \int_{0}^{T}
e^{-\lambda t}\big(1+ \ds \int_{0}^{t}
(t-s)^{-\frac{k}{\alpha}} \mathbb{E}|u(s)|_{2}^{p}ds\big)dt\nonumber\\
&\leq & K \big(1+ \int_{0}^{T}\big(\int_{0}^{T-s}e^{-\lambda \tau
}\tau^{-\frac{k}{\alpha}}d\tau\big)e^{-\lambda
s}\mathbb{E}|u(s)|_{2}^{p}ds\big)\nonumber\\
&\leq & K \big(1+ \int_{0}^{T}e^{-\lambda
s}\mathbb{E}|u(s)|_{2}^{p}ds\big)< \infty.
\end{eqnarray}
Therefore, $\mathcal{A}_{k+1}u \in \mathcal{H}^*, \forall k=0,1,...m
$. To estimate the stochastic integral, we use the factorization
method (see \cite{DaZa}). By the semigroup property and the Fubini's
Theorem and the identity
$$
\int_{\sigma}^{t}(t-s)^{\beta-1}(s-\sigma)^{-\beta}ds=
\frac{\pi}{\sin \pi \beta}, \qquad \sigma\leq s \leq t, \quad
0<\beta<1,
$$
we get the following representation of $ \mathcal{A}_{m+2}u(t, x) $
$$
\mathcal{A}_{m+2}u(t, x)=
\frac{\sin\pi\beta}{\pi}\int_{0}^{t}(t-s)^{\beta-1}\int_{\mathbb{R}}{}_{\delta}G_{\alpha}(t-s,x-y)Y(s,y)dyds,
$$
where
$$
Y(s,y)=
\int_{0}^{s}(s-\sigma)^{-\beta}\int_{\mathbb{R}}{}_{\delta}G_{\alpha}(s-\sigma,y-z)f(\sigma,
z, u(\sigma, z))W(dzd\sigma).
$$
Let $ \frac{1}{\alpha p}- \frac{1}{2\alpha}< \beta <1 $ . By the
inequalities cited above and the corollary \ref{Corollary}, we get
\begin{eqnarray}\label{Am+2Y(s)}
\mathbb{E}|\mathcal{A}_{m+2}u(t)|_{2}^p& =& \mathbb{E}\Big(\ds
\int_{\mathbb{R}}\ds \Big|\frac{\sin\pi\beta}{\pi}\int_{0}^{t}\ds
\ds(t-s)^{\beta-1}Y(s,.)*{}_{\delta}G_{\alpha}(t-s,.)ds \Big|^{2}dx\Big)^{\frac{p}{2}}\nonumber\\
& \leq & \mathbb{E}\Big(\frac{|\sin\pi\beta|}{\pi}\ds
\int_{0}^{t}(t-s)^{\beta-1} \ds |Y(s,.)*
{}_{\delta}G_{\alpha}(t-s,.)|_{2}ds\Big)^p.\nonumber\\
& \leq & \mathbb{E}\Big(\frac{|\sin\pi\beta|}{\pi}\ds
\int_{0}^{t}(t-s)^{\beta-1} \ds |Y(s)|_{p}
|{}_{\delta}G_{\alpha}(t-s,.)|_{\frac{2p}{3p-2}}ds\Big)^p\nonumber\\
& \leq & K \mathbb{E}\Big(\ds \int_{0}^{t}(t-s)^{\beta-1+
\frac{1}{2\alpha}- \frac{1}{\alpha p}}
\ds |Y(s)|_{p}ds\Big)^p\nonumber\\
& \leq & K \ds \int_{0}^{t}(t-s)^{\beta-1+
\frac{1}{2\alpha}-\frac{1}{\alpha p}}\ds \mathbb{E}|Y(s)|_{p}^{p}ds.\nonumber\\
\end{eqnarray}
On the other hand, under the condition $0 < \beta< \frac{1}{2}+
\frac{1}{\alpha p}-\frac{1}{\alpha} $ and using
Burkholder-Davis-Gundy inequality and the generalized  Minkowsky's
inequality, Young's inequality, corollary \ref{Corollary},
conditions (\ref{Lipschitz Cond.}) and H\"{o}lder inequality, we get
for all $ 0\leq s\leq t\leq T $
\begin{eqnarray*}
\mathbb{E}\big[|Y(s)|_{p}^{p}\big]& \leq & \ds\int_{\mathbb{R}}\ds
\mathbb{E}\ds \sup_{0 \leq \tau \leq
s}\Big|\int_{0}^{\tau}\ds\int_{\mathbb{R}}\ds
\ds(s-\sigma)^{-\beta}{}_{\delta}G_{\alpha}(s-\sigma,y-z)f(\sigma, z, u(\sigma, z))W(dzd\sigma) \Big|^{p}dy\\
& \leq & K \ds\int_{\mathbb{R}}\ds \mathbb{E}\ds
\Big|\int_{0}^{s}\ds\int_{\mathbb{R}}\ds
\ds(s-\sigma)^{-2\beta}{}_{\delta}G_{\alpha}^{2}(s-\sigma,y-z)f^{2}(\sigma,
z, u(\sigma, z))dzd\sigma
\Big|^{\frac{p}{2}}dy\\
& \leq & K \ds \mathbb{E}\ds\int_{\mathbb{R}}\ds
\Big|\int_{0}^{s}\ds(s-\sigma)^{-2\beta}({}_{\delta}G_{\alpha}^{2}(s-\sigma,.)*f^{2}(\sigma,
., u(\sigma, .)))(y)d\sigma
\Big|^{\frac{p}{2}}dy\\
& \leq & K \ds \mathbb{E}\Big(\ds
\int_{0}^{s}\ds(s-\sigma)^{-2\beta}\big|{}_{\delta}G_{\alpha}^{2}(s-\sigma,.)*f^{2}(\sigma,
., u(\sigma, .))\big|_{\frac{p}{2}}d\sigma
\Big)^{\frac{p}{2}}\\
& \leq & K \ds \mathbb{E}\Big(\ds
\int_{0}^{s}\ds(s-\sigma)^{-2\beta}\big|{}_{\delta}G_{\alpha}^{2}(s-\sigma,.)\big|_{\frac{p}{2}}\big|f^{2}(\sigma,
., u(\sigma, .))\big|_{1}d\sigma
\Big)^{\frac{p}{2}}\\
& \leq & K\ds \mathbb{E}\Big(\ds \int_{0}^{s}\ds(s-\sigma)^{-2\beta
+ \frac{2}{\alpha p}- \frac{2}{\alpha}}\big(|a_{m+1}|_{2}^{2}+
|u(\sigma)|_{2}^{2}\big)d\sigma \Big)^{\frac{p}{2}}\\
& \leq & K\ds \Big(1 + \ds \int_{0}^{s}\ds(s-\sigma)^{-2\beta +
\frac{2}{\alpha p}-
\frac{2}{\alpha}}\mathbb{E}|u(\sigma)|_{2}^{p}d\sigma \Big).\\
\end{eqnarray*}
Replacing this last inequality in (\ref{Am+2Y(s)}), we get
\begin{eqnarray}\label{A2(m+2)u}
\mathbb{E}|\mathcal{A}_{m+2}u(t)|_{2}^{p}& \leq & K \ds
\int_{0}^{t}(t-s)^{\beta-1+\frac{1}{2\alpha}-\frac{1}{\alpha
p}}\ds \mathbb{E}|Y(s)|_{p}^{p}ds\nonumber\\
& \leq & K \Big(1+ \ds
\int_{0}^{t}(t-s)^{\beta-1+\frac{1}{2\alpha}-\frac{1}{\alpha
p}}\big(\ds \int_{0}^{s} (s-\sigma)^{-2\beta+
\frac{2}{\alpha P}- \frac{2}{\alpha}}\mathbb{E}|u(\sigma)|_{2}^{p}d\sigma \big)ds\Big)\nonumber\\
& \leq & K \Big(1+ \ds \int_{0}^{t}(t-\sigma)^{-\beta+
\frac{1}{\alpha p}- \frac{3}{2\alpha}}\mathbb{E}|u(\sigma)|_{2}^{p}d\sigma\Big).\nonumber\\
\end{eqnarray}
But $ -\beta-\frac{3}{2\alpha}+\frac{1}{\alpha p} >-1 $ because $
1-\frac{3}{2\alpha}+\frac{1}{\alpha p}> \frac{1}{2}+\frac{1}{\alpha
p}- \frac{1}{\alpha}$. Hence $\{\mathcal{A}_{m+2}u(t), t\geq 0\}$ is
an $L^2-$ process satisfying (\ref{UnifBoundIneqSolu}). Further,
$$\int_{0}^{T}e^{-\lambda
t}\mathbb{E}|\mathcal{A}_{m+2}u(t)|_{2}^{p}dt \leq  K (1+
\int_{0}^{T} e^{-\lambda
\sigma}\mathbb{E}|u(\sigma)|_{2}^{p}d\sigma)<\infty.$$ It is clear
that $ \frac{1}{\alpha p}- \frac{1}{2\alpha }<
\frac{1}{2}+\frac{1}{\alpha p}- \frac{1}{\alpha}$, provided that $
\alpha>1$. The parameter $ \beta $ satisfies the inequality
$\max\{0, \frac{1}{\alpha p}- \frac{1}{2\alpha}\}< \beta<\min\{
\frac{1}{2}+ \frac{1}{\alpha p}-\frac{1}{\alpha}, 1\} $. This
achieves the proof that $ \mathcal{A}u \in \mathcal{H}^*$.

To prove the existence and the uniqueness of the solution in
$\mathcal{H}^*$, we use the fixed point method. Let $ u, v \in
\mathcal{H}^* $, we have for all $ k = 0, 1, ... m $, for all $t>0$,
\begin{eqnarray*}
\mathbb{E}|\mathcal{A}_{k+1}u(t) - \mathcal{A}_{k+1}v(t)|_{2}^p&=&
\mathbb{E}\Big(\int_{\mathbb{R}}\big|\int_{0}^{t}{}_{\delta}G_{\alpha}^{(k)}(t-s,.)*\big(
h_{k}(s, u(s))- h_{k}(s,
v(s))\big)ds\big|^{2}dx\Big)^{\frac{p}{2}}\\
&\leq& \mathbb{E}\big(\int_{0}^{t}
|{}_{\delta}G_{\alpha}^{(k)}(t-s,.)*\big( h_{k}(s, u(s))-
h_{k}(s,v(s))\big)|_{2}ds\big)^p\\
&\leq & \mathbb{E}\big(\int_{0}^{t}
|{}_{\delta}G_{\alpha}^{(k)}(t-s,.)|_{1} \big| h_{k}(s,
u(s))- h_{k}(s,v(s))\big|_{2}ds\big)^p\\
&\leq & K\mathbb{E}\big(\int_{0}^{t}(t-s)^{-\frac{k}{\alpha}}
\big|u(s)- v(s)\big|_{2}ds\big)^p\\
&\leq & K\int_{0}^{t}(t-s)^{-\frac{m}{\alpha}}\mathbb{E} \big|u(s)-
v(s)\big|_{2}^{p}ds.
\end{eqnarray*}
Therefore
\begin{eqnarray}\label{ A(k)u-A(k)v}
|\mathcal{A}_{k+1}u - \mathcal{A}_{k+1}v|^{p}_{\mathcal{H}}& \leq &
K\int_{0}^{T}e^{-\lambda s}\Big[\int_{0}^{T}e^{-\lambda
\tau}\tau^{-\frac{m}{\alpha}}d\tau\Big] \mathbb{E}\big|u(s)-
v(s)\big|_{2}^{p}ds \nonumber \\
& \leq &
K\lambda^{\frac{m-\alpha}{\alpha}}\Gamma(\frac{\alpha-m}{\alpha})
\big|u - v \big|^{p}_{\mathcal{H}}.\nonumber \\
\end{eqnarray}
For the stochastic integral, we use again the factorization method
for the same $ \beta $. We get
$$
\mathcal{A}_{m+2}u(t,x) - \mathcal{A}_{m+2}v(t,x)=
\frac{\sin\pi\beta}{\pi}\int_{0}^{t}(t-s)^{\beta-1}\int_{\mathbb{R}}{}_{\delta}G_{\alpha}(t-s,x-y)\zeta(s,y)dyds,
$$
where
$$
\zeta(s,y)=
\int_{0}^{s}(s-\sigma)^{-\beta}\int_{\mathbb{R}}{}_{\delta}G_{\alpha}(s-\sigma,y-z)\big(
f(\sigma, z, u(\sigma, z))- f(\sigma, z, v(\sigma,
z))\big)W(dzd\sigma).
$$
Using the calculus above, we obtain
$$
\mathbb{E}|\mathcal{A}_{m+2}u(t) - \mathcal{A}_{m+2}v(t)|_{2}^{p}
\leq K \int_{0}^{t}(t-s)^{\beta-1-\frac{1}{\alpha}+
\frac{3p-2}{2\alpha p}}\mathbb{E}|\zeta (s)|_{p}^{p}ds,
$$
and
\begin{eqnarray*}
\mathbb{E}|\zeta (s)|_{p}^{p}& \leq &K\ds \mathbb{E}
\Big(\int_{0}^{s}\ds(s-\sigma)^{-2\beta}\big|{}_{\delta}G_{\alpha}^{2}(s-\sigma,.)\big|_{\frac{p}{2}}\big|
\big(f(\sigma, ., u(\sigma, .))- f(\sigma, ., v(\sigma,
.))\big)^{2}\big|_{1}d\sigma \Big)^{\frac{p}{2}}\\
& \leq & K\ds \int_{0}^{s}\ds(s-\sigma)^{-2\beta + \frac{2}{\alpha
p} - \frac{2}{\alpha}}\mathbb{E}\big|u(\sigma)-
v(\sigma)\big|_{2}^{p}d\sigma.
\end{eqnarray*}
Hence
\begin{eqnarray*}
\mathbb{E}|\mathcal{A}_{m+2}u(t) - \mathcal{A}_{m+2}v(t)|_{2}^{p} &
\leq & K\int_{0}^{t}(t-s)^{\beta-1-\frac{1}{\alpha}+
\frac{3p-2}{2\alpha p}}\int_{0}^{s}\ds(s-\sigma)^{-2\beta +
\frac{2}{\alpha p} - \frac{2}{\alpha}}\mathbb{E}\big|u(\sigma)-
v(\sigma)\big|_{2}^{p}d\sigma ds\\
& \leq & K\int_{0}^{t}(t-\sigma)^{-\beta+\frac{1}{\alpha p}-
\frac{3}{2\alpha}}\mathbb{E}\big|u(\sigma)-
v(\sigma)\big|_{2}^{p}d\sigma,\
\end{eqnarray*}
and therefore
\begin{eqnarray}\label{ A2(m+1)u-A(m+1)v}
|\mathcal{A}_{m+2}u - \mathcal{A}_{m+2}v|^{p}_{\mathcal{H}}& \leq &
K\int_{0}^{T}e^{-\lambda s}\Big[\int_{0}^{T}e^{-\lambda
\tau}\tau^{^{-\beta+\frac{1}{\alpha p}-
\frac{3}{2\alpha}}}d\tau\Big] \mathbb{E}\big|u (s)
- v(s) \big|_{2}^{p}ds \nonumber \\
& \leq & K\lambda^{\beta-\frac{1}{\alpha p}+
\frac{3}{2\alpha}-1}\Gamma(1-\beta + \frac{1}{\alpha p}-
\frac{3}{2\alpha})
\big|u - v \big|_{\mathcal{H}}^{p}.\nonumber \\
\end{eqnarray}
For a good choice of the constant $ \lambda $ such that $
\max\{(m+1)K\lambda^{\beta-\frac{1}{\alpha p}+
\frac{3}{2\alpha}-1}\Gamma(1-\beta+\frac{1}{\alpha p}-
\frac{3}{2\alpha}),
K\lambda^{\frac{m-\alpha}{\alpha}}\Gamma(\frac{\alpha-m}{\alpha})\}
< 1 $, the operator $ \mathcal{A} $ is then a contraction. This
choice is possible because the exponent $\beta-\frac{1}{\alpha p}+
\frac{3}{2\alpha}-1$ is also negative. So there exists an unique $
L^{2}(\mathbb{R})-$valued $ \mathcal{F}_{t}-$adapted process $ u \in
\overline{\mathcal{H}^*}\subset \mathcal{H} $ solution of Equation
(\ref{Eq.Mild.Solu}), where $\overline{\mathcal{H}^*}$ is the
closure of the subspace $\mathcal{H}^* $ in $\mathcal{H}$. We prove
now that $ u \in \mathcal{H}^* $ i.e.  $ u$ satisfies
(\ref{UnifBoundIneqSolu}). In fact, we have on one hand $
\mathcal{A}u=u$, on the other hand by a similar calculus of that
done above, we get for all $ t > 0$, the inequalities
(\ref{IneqA0u(t)}), (\ref{Est.give expect Ak+1Intermd}) and
(\ref{A2(m+2)u}). Hence
\begin{eqnarray}\label{ReccuFormEu}
\mathbb{E}|u(t)|_{2}^{p} &\leq &K \Big(1+ \ds
\int_{0}^{t}\Big[\sum_{k=0}^{m} (t-\sigma)^{-\frac k\alpha} +
(t-\sigma)^{-\beta+ \frac{1}{\alpha p}-
\frac{3}{2\alpha}}\Big]\mathbb{E}|u(\sigma)|_{2}^{p}d\sigma\Big)\nonumber\\
&\leq &K \Big(1+ \ds \int_{0}^{t}
(t-\sigma)^{-\gamma}\mathbb{E}|u(\sigma)|_{2}^{p}
d\sigma\Big),\nonumber\\
\end{eqnarray}
where $\gamma= \min\{\frac m\alpha, \beta- \frac{1}{\alpha p}+
\frac{3}{2\alpha}\}$. By Gronwall Lemma  we get $
\mathbb{E}|u(t)|_{2}^{p} \leq K_\gamma e^{K_\gamma T} $, hence $u$
satisfies (\ref{UnifBoundIneqSolu}). We prove now the uniqueness of
the solution with respect to the sup norm. Let $ u_1, u_2$ be two
solutions of (\ref{Eq.Mild.Solu}) in $\mathcal{H}^*$. By a similar
calculus as above, we obtain for all $t>0$
\begin{eqnarray*}
\mathbb{E}|u_1(t)- u_2(t)|_{2}^{p} &\leq &K_\alpha  \int_{0}^{t}
\mathbb{E}|u_1(\sigma)- u_2(\sigma)|_{2}^{p}
d\sigma.\nonumber\\
\end{eqnarray*}
Again by Gronwall Lemma we get $\sup_{[0, T]}\mathbb{E}|u_1(t)-
u_2(t)|_{2}^{p}=0$ ( see e.g. \cite{W1} page 314 for a similar
calculus for $\alpha=2$ ).
\end{proof}

\begin{coro}

Let $ \alpha >1$, $ u^{0} \in L^{p}(\Omega, \mathcal{F}_{0},
L^{2})$, $ p\geq 1 $ such that $ \sup_y\mathbb{E}|u^{0}(y)|^q<\infty
$ for some $q\geq 2$, and let the coefficients $f, h_k\;
k=0,1,...m+1 $ satisfy the conditions (\ref{Lipschitz Cond.}) such
that $ a_k \in L_\infty (\mathbb{R})\cap L^2(\mathbb{R}), \;
k=0,1,...m+1$. Then the equation in (\ref{Eq.Principal}) admits a
unique $L^2-$mild solution which is continuous in space and in time
and satisfies the inequality
\begin{equation}
\max\{\sup_{[0, T]}\mathbb{E}|u(s)|_{2}^p, \sup_{[0,
T]}\sup_y\mathbb{E}|u(s, y)|^q \}<\infty.
\end{equation}
\end{coro}
The corollary follows from Theorem 1 in \cite{DD1} and Theorem
\ref{Exist.mild.solu.} in this paper.

\section{Equivalence of Solutions}

Let us consider two kinds of solutions of variational type of the
SFPDE in (\ref{Eq.Principal}), for which the coefficients $ f, h_k $
satisfy the conditions (\ref{Lipschitz Cond.}) and the initial
condition $u^0 \in L^p(\Omega, \mathcal{F}_0, L^2)$.
\begin{defn}\label{Def.Weak.Sol.1}
A $ L^{2}-$valued $ \mathcal{F}_t-$ adapted stochastic process $u
=\{ u(t,.), t\in [0,T]\}$ is said to be a weak solution of the first
kind on the interval $ [0,T] $ of Equation (\ref{Eq.Principal}), if
$u $ satisfies the assumption (\ref{UnifBoundIneqSolu}) and the
following integral equation, for all $ t \in [0, T] $ and for all $
\phi \in C^{\infty}_{0} $, where $ \phi \in C^{\infty}_{0} $ is the
set of infinitely differentiable functions with compact support on $
\mathbb{R}$:
\begin{equation}\label{Weak.Eq.1}
\begin{array}{rl}
\ds \int_{\mathbb{R}}\ds u(t,x)\phi(x)dx = & \ds\int_{\mathbb{R}}\ds
u^{0}(x)\phi(x)\ dx  +  \ds \int_{0}^{t}\ds\int_{\mathbb{R}}\ds
u(s,x) {}_{x}D^{\alpha}_{-\delta} \phi(x)
dx ds \\
+ & \sum_{k=0}^{m}(-1 )^{k}\ds \int_{0}^{t}\ds \ds \int_{\mathbb{R}}
\ds \ds h_{k}(s,x,u(s,x))
\phi^{(k)}(x)\ dxds \\
 & + \ds
\int_{0}^{t}\int_{\mathbb{R}}f(s,x,u(s,x))\phi(x)W(dxds)\; \; a.s.
\end{array}
\end{equation}
\end{defn}

\begin{defn}\label{Def.Weak.Sol.2}
A $ L^{2}-$valued $ \mathcal{F}_t-$ adapted stochastic process $u
=\{ u(t,.), t\in [0,T]\}$ is said to be a weak solution of the
second kind on the interval $ [0,T] $ of Equation
(\ref{Eq.Principal}), if $u $ satisfies the assumption
(\ref{UnifBoundIneqSolu}) and the following integral equation, for
all $ t \in [0, T] $ and for all $  \psi \in C^{1, \infty}((0,
t)\times \mathbb{R})$ and such that $\psi(s, .) \in
D(D^{\alpha}_{\delta}), \forall s< t$:
\begin{equation}\label{Var.Eq}
\begin{array}{rl}
\ds \int_{\mathbb{R}}\ds u(t,x)\psi(t, x)dx = &
\ds\int_{\mathbb{R}}\ds u^{0}(x)\psi(0, x)\ dx +
\int_{0}^{t}\int_{\mathbb{R}}u(s,x)
\partial_{s} \psi(s, x)
dx ds \\
+ & \ds \int_{0}^{t}\ds\int_{\mathbb{R}}\ds u(s,x)
{}_{x}D^{\alpha}_{-\delta} \psi(s, x)
dx ds \\
+ & \sum_{k=0}^{m}(-1 )^{k}\ds \int_{0}^{t}\ds \ds \int_{\mathbb{R}}
\ds \ds h_{k}(s,x,u(s,x))
\partial_{x}^{(k)} \psi(s,x)\ dxds \\
 & + \ds
\int_{0}^{t}\int_{\mathbb{R}}f(s,x,u(s,x))\psi(s,x)W(dxds)\;\; a.s.
\end{array}
\end{equation}
\end{defn}

\begin{theorem}\label{Theo.Equivalence Solu.}
For $p\geq 2$, the different notions of solutions given in
Definitions \ref{Def.mild.Solu.}, \ref{Def.Weak.Sol.1} and
\ref{Def.Weak.Sol.2} are equivalent.
\end{theorem}

\begin{proof}
We prove the equivalence between Definition \ref{Def.mild.Solu.} and
Definition \ref{Def.Weak.Sol.2}. The other equivalences are obtained
by a similar way ( see \cite{GongyPrint}, for $ \alpha =2$).

Let $ u= \{u(t) , t \in [0, T]\} $ be a weak solution of second kind
of Equation (\ref{Eq.Principal}) and let $ \phi(.) \in
C^{\infty}_{0}(\mathbb{R})$. We define the function $ \psi^{t}(s,
x)$ by
$$
\psi^{t}(s, x) = \Big\{
\begin{array}{lr}
\phi(x), \mbox{ when } t = s,\\
\int_{\mathbb{R}}{}_{-\delta}G_{\alpha}(t-s, x-y)\phi(y)dy,\mbox{
when } s < t,
\end{array}
$$
The function $ \psi^{t}(.,.) \in C^{1, \infty}((0, t)\times
\mathbb{R})$ and we have for all fixed $ s < t$, $
\mathcal{F}\{\psi^{t}(s, x), \lambda\}=
e^{{}_{-\delta}\psi_{\alpha}(\lambda)(t-s)}\hat{\phi}(\lambda). $
Hence $ \psi^{t}(s, x) \in D({}_{x}D^{\alpha}_{-\delta})$ and $
{}_{x}D^{\alpha}_{-\delta}\psi^{t}(s, x)=
\mathcal{F}^{-1}\{{}_{-\delta}
\psi_{\alpha}(\lambda)e^{{}_{-\delta}\psi_{\alpha}(\lambda)(t-s)}\hat{\phi}(\lambda),
x\}$. On the other hand for fixed $ x $, $ \psi^{t}(s, x) $ is
differentiable with respect to $ s<t $, because $
{}_{-\delta}G_{\alpha}(t-s, z) $ is differentiable with respect to $
s $ and  $\delta_s {}_{-\delta}G_{\alpha}(t-s, .) \phi(.)$ is
integrable. Further
\begin{eqnarray*}
\partial_{s}\psi^{t}(s,x) &=& \int_{\mathbb{R}}\partial_{s}{}_{-\delta}G_{\alpha}(t-s,
x-y)\phi(y)dy\\
&=& -\int_{\mathbb{R}}\partial_{\tau}{}_{-\delta}G_{\alpha}(\tau,
y)|_{\tau= t-s}\phi(x-y)dy\\
&=&
-\int_{\mathbb{R}}{}_{y}D^{\alpha}_{-\delta}{}_{-\delta}G_{\alpha}(t-s,
y)\phi(x-y)dy\\
&=&-\mathcal{F}^{-1}\{{}_{-\delta}\psi_{\alpha}(\lambda)e^{{}_{-\delta}\psi_{\alpha}(\lambda)(t-s)}
\hat{\phi}(\lambda), x\}\\
&=& - {}_{x}D^{\alpha}_{-\delta}\psi^{t}(s, x).
\end{eqnarray*}
We replace $ \psi(s , x) $ by $ \psi^{t}(s , x) $ in equation
(\ref{Var.Eq}), apply deterministic and stochastic Fubini's Theorems
and the fact that $ {}_{-\delta}G_{\alpha}(t-s, x-y) =
{}_{\delta}G_{\alpha}(t-s, y-x)$. We interpret the integrals on $
\mathbb{R}$ as the scalar product in $ L^{2}(\mathbb{R}) $ and use
estimates as in section 2 to prove that $ \mathcal{A}_{k}u(t) \in
L^{2}(\mathbb{R})\;  a.s., \forall k\in \{0, 1, ..., m+2\} $, we get
$ \langle u(t)-( \mathcal{A}_{0}+ \sum_{k=0}^{m}\mathcal{A}_{k+1}-
\mathcal{A}_{m+2})u(t), \phi \rangle_{L^{2}}= 0 $. Since $
C^{\infty}_{0}$ is dense in $L^2(\mathbb{R})$, we obtain that $ u=
\{u(t) , t \in [0, T]\} $ is a mild solution of Equation
(\ref{Eq.Principal}). Fubini's Theorems are applied thanks to
$$
\int_{\mathbb{R}}\int_{\mathbb{R}}|u^{0}(x)||{}_{\delta}G_{\alpha}(t-s,
y-x)||\phi(y)| dy dx \leq
|u^{0}|_{2}|\phi|_{2}|{}_{\delta}G_{\alpha}(1,.)|_{2}< \infty \; \;
\; a.s,
$$
\begin{eqnarray*}
\mathbb{E}\Big(\int_{0}^{t}\int_{\mathbb{R}}\int_{\mathbb{R}}|h_{k}(s,
x, u(s,x))|
\!\!\!\!\!&|&\!\!\!\!\!\!{}_{-\delta}G_{\alpha}^{(k)}(t-s, x-y)||\phi(y)| dy dx ds \Big)^p\\
&\leq &\ds\mathbb{E}\Big(\int_{0}^{t}\ds |h_{k}(s,., u(s,.))|_{2}
||{}_{-\delta}G_{\alpha}^{(k)}(t-s, .)|*|\phi||_{2}ds\Big)^p\\
&\leq& K |\phi|_{2}^p\mathbb{E}\Big(1 +
\int_{0}^{t}(t-s)^{-\frac{k}{\alpha}}|u(s)|_{2}ds\Big)^p\\
&\leq& K |\phi|_{2}^p\mathbb{E}\Big(1 +
\int_{0}^{t}(t-s)^{-\frac{k}{\alpha}}|u(s)|_{2}^pds\Big)\\
&\leq& K (1 + \sup_{[0, T]}\mathbb{E}|u(s)|_{2}^{p})<\infty,\\
\end{eqnarray*}
\begin{eqnarray*}
\int_{0}^{t}\int_{\mathbb{R}}\mathbb{E}\Big(\int_{\mathbb{R}}\phi(y)f(s,
x, u(s,
x)\!\!\!\!\!\!\!\!&\!\!\!)\!\!\!\!&\!\!\!\!\!\!\!\!\!{}_{-\delta}G_{\alpha}(t-s,
x-y)dy\Big)^{2}dxds\\
&\leq& K|\phi|_{\infty}^2|{}_{-\delta}G_{\alpha}(1,
.)|_{1}^2\mathbb{E}(1 + \int_{0}^{t}|u(s)|_{2}^{2}ds)\\
&\leq& K\mathbb{E}(1 + \int_{0}^{t}|u(s)|_{2}^{2}ds).\\
\end{eqnarray*}
Using Jensen inequality, we get
\begin{eqnarray*}
\Big(\int_{0}^{t}\int_{\mathbb{R}}\mathbb{E}\Big(\int_{\mathbb{R}}\phi(y)f(s,
x, u(s,
x)\!\!\!\!\!\!\!\!&\!\!\!)\!\!\!\!&\!\!\!\!\!\!\!\!\!{}_{-\delta}G_{\alpha}(t-s,
x-y)dy\Big)^{2}dxds\Big)^\frac p2\\
&\leq& K(1 + \int_{0}^{t}\mathbb{E}|u(s)|_{2}^{p}ds)<\infty.
\end{eqnarray*}

Let now $u= \{ u(t), t\geq 0\}$  be a mild solution of Equation
(\ref{Eq.Principal}). From Theorem \ref{Exist.mild.solu.}, $u$
satisfies the inequality (\ref{UnifBoundIneqSolu}). Furthermore, for
fixed $t>0$, let $  \psi \in C^{1, \infty}([0, t]\times \mathbb{R})$
such that $\psi(s, .) \in D(D^{\alpha}_{\delta}), \forall s< t$. by
replacing the right hand side of (\ref{Eq.Mild.Solu}) in the left
hand side of (\ref{Var.Eq}), we get
\begin{eqnarray}\label{Eq.mild to weak N1}
\int_{\mathbb{R}}u(t, x)\psi(t, x) dx &=& \int_{\mathbb{R}}\psi(t,
x)\big(\int_{\mathbb{R}}{}_{\delta}G_{\alpha}(t,
x-y)u^{0}(y)dy\big)dx\nonumber\\
&+& \sum_{k=0}^{m}(-1)^{k}\int_{\mathbb{R}}\psi(t,
x)\big(\int_{0}^{t}\int_{\mathbb{R}}{}_{\delta}G_{\alpha}^{(k)}(t-s,
x-y)h_{k}(s, y, u(s, y))dyds\big)dx\nonumber\\
&+&\int_{\mathbb{R}}\psi(t,
x)\big(\int_{0}^{t}\int_{\mathbb{R}}{}_{\delta}G_{\alpha}(t-s,
x-y)f(s, y, u(s, y))W(dyds)\big)dx.\nonumber\\
\end{eqnarray}
Applying Fubini's Theorems (deterministic and stochastic) to each
term on the right hand of (\ref{Eq.mild to weak N1}), we get the
terms $ \int_{\mathbb{R}}{}_{\delta}G_{\alpha}^{(k)}(t-s,
x-y)\psi(t, x)dx, k=0,1,...m $ in the right hand side of
(\ref{Eq.mild to weak N1}). Using the properties of Green's function
$ {}_{\delta}G_{\alpha}(t, x) $ and the integral by parts (see
\cite{D2}), we obtain
\begin{eqnarray}\label{Eq.1 From mild to weak}
\int_{\mathbb{R}}{}_{\delta}G_{\alpha}(t-s, x-y)\psi(t, x)dx &=&
\psi(s, y)+
\int_{\mathbb{R}}\int_{s}^{t}\frac{\partial}{\partial\sigma}({}_{\delta}G_{\alpha}(\sigma-s,
x-y)\psi(\sigma, x))dxd\sigma \nonumber\\
&=&\psi(s, y) + \int_{s}^{t}\int_{\mathbb{R}}
{}_{\delta}G_{\alpha}(\sigma-s,
x-y){}_{x}D^{\alpha}_{-\delta}\psi(\sigma, x)dxd\sigma   \nonumber\\
&+&\int_{s}^{t}\int_{\mathbb{R}} {}_{\delta}G_{\alpha}(\sigma-s,
x-y)\partial_{\sigma} \psi(\sigma, x)dxd\sigma.\nonumber\\
\end{eqnarray}
Furthermore, by the commutativity of the operators $
{}_{x}D^{\alpha}_{-\delta} $ and $ D^{k}$, where this last is the
classical differential operator of entire order $ k $, we get
\begin{eqnarray}\label{Eq.Prticular deriv. phi}
\int_{\mathbb{R}}{}_{\delta}G_{\alpha}^{(k)}(t-s, x-y)\psi(t, x)dx
&=& (-1)^k\int_{\mathbb{R}}{}_{\delta}G_{\alpha}(t-s,
x-y)\psi^{(k)}_x(t,
x)dx\nonumber\\
&=&\psi^{(k)}_{y}(s, y) + \int_{s}^{t}\int_{\mathbb{R}}
{}_{\delta}G_{\alpha}^{(k)}(\sigma-s,
x-y){}_{x}D^{\alpha}_{-\delta}\psi(\sigma, x)dxd\sigma   \nonumber\\
&+&\int_{s}^{t}\int_{\mathbb{R}}
{}_{\delta}G_{\alpha}^{(k)}(\sigma-s,
x-y)\partial_{\sigma} \psi(\sigma, x)dxd\sigma.\nonumber\\
\end{eqnarray}
By replacing (\ref{Eq.1 From mild to weak}) and (\ref{Eq.Prticular
deriv. phi}) in to corresponding terms in (\ref{Eq.mild to weak N1})
and again by applying Fubini's Theorem, we get
\begin{eqnarray}
\int_{\mathbb{R}}\psi(t,x)\big(\int_{\mathbb{R}}{}_{\delta}G_{\alpha}(t,
x-y)u^{0}\!\!\!\!\!\!\!\!\!&{}& \!\!\!\!\!\!\!(y)dy\big)dx = \int_{\mathbb{R}}u^{0}(y)\psi(0,y) dy\nonumber \\
\!\!\!\!\!\!\!&+&\!\!\!\! \!\!\!\int_{0}^{t}\int_{\mathbb{R}}\Big(
\int_{\mathbb{R}}{}_{\delta}G_{\alpha}(\sigma,
x-y)u^{0}(y)dy\Big)({}_{x}D^{\alpha}_{-\delta}\psi(\sigma, x)+ \partial_{\sigma}\psi(\sigma, x))dxd\sigma, \nonumber \\
\end{eqnarray}
\begin{eqnarray}
\int_{\mathbb{R}}\psi(t,x)\!\!\!\!\!\!\!\!\!\!\!&{}&\!\!\!\!\!\big(\int_{0}^{t}\int_{\mathbb{R}}{}_{\delta}G_{\alpha}(t-s,
x-y)f(s, y, u(s, y)))W(dyds)\big)dx=\int_{0}^{t}\int_{\mathbb{R}}f(s, y, u(s, y))\psi(s, y)W(dyds)\nonumber\\
&+&\int_{0}^{t}\int_{\mathbb{R}}\big(\int_{0}^{\sigma}\int_{\mathbb{R}}{}_{\delta}G_{\alpha}(\sigma-s,
x-y)f(s, y, u(s, y)))W(dyds)\big)({}_{x}D^{\alpha}_{-\delta}\psi(\sigma, x) + \partial_{\sigma}\psi(\sigma, x))dxd\sigma, \nonumber\\
\end{eqnarray}
\begin{eqnarray}
\int_{\mathbb{R}}\psi(t,x)\!\!\!\!\!\!\!\!\!\!\!&{}&\!\!\!\!\!\!\!\big(\int_{0}^{t}\int_{\mathbb{R}}
{}_{\delta}G_{\alpha}^{(k)}(t-s,
x-y)h_{k}(s, y, u(s, y))dyds\big)dx
=\int_{0}^{t}\int_{\mathbb{R}}h_{k}(s, y, u(s,
y))\psi^{(k)}_{y}(s, y)dyds\nonumber\\
&+&\int_{0}^{t}\int_{\mathbb{R}}\big(\int_{0}^{\sigma}\int_{\mathbb{R}}{}_{\delta}G_{\alpha}^{(k)}
(t-s, x-y)h_{k}(s, y, u(s,
y))dy ds\big)({}_{x}D^{\alpha}_{-\delta}\psi(\sigma, x)+\partial_{\sigma}\psi(\sigma, x))dxd\sigma.\nonumber\\
\end{eqnarray}

Replacing these equalities in (\ref{Eq.mild to weak N1}) and using
the fact that $ u(t, .) $ is a mild solution, we get that $ u(t, .)
$ satisfies Equation (\ref{Var.Eq}). We can check as in the first
part of the prove that Fubini's Theorems can be applied.

\end{proof}
As a consequence of Theorems \ref{Exist.mild.solu.} and
\ref{Theo.Equivalence Solu.}, we have
\begin{coro}
Equation (\ref{Eq.Principal}) for with the coefficients $ f, h_k $
satisfying the conditions (\ref{Lipschitz Cond.}) and the initial
condition $u^0 \in L^p(\Omega, \mathcal{F}_0, L^2)$, $p\geq 2$
admits a unique weak solution of first kind and of second kind.
\end{coro}


\section{Application of Fourier transform in a random SFPDE} We consider the following SFPDE obtained
from Equation (\ref{Eq.Principal}) by taking $ h_{k}(t, x, r)= c_{k}
r $, for all $ 0\leq k\leq m $ and where $ f $ may be random but
independent of the solution $ u $
\begin{equation}\label{Eq.Princ.Four}
\frac{\partial }{\partial s}u(s, y)= {}_{x}D^{\alpha}_{\delta}u(s,
y)+ \sum_{0}^{m}c_{k}\frac{\partial^{k}}{\partial y^{k}}u(s, y)+
f(s, y)\frac{\partial^{2} W}{\partial s\partial y}(y, s).
\end{equation}
Multiplying the two sides of the above equation by $ e^{i\lambda y}$
and integrating with respect to $ s $ and to $ y$,we get
\begin{equation}\label{Eq.u.Hat.Integ.Form.Appro.Special}
\hat{u}(t,\lambda)= \hat{u^{0}}(\lambda)+
({}_{\delta}\psi_{\alpha}(\lambda)+ \sum_{0}^{m}c_{k}(-i\lambda)^{k}
)\int_{0}^{t}\hat{u}(s,\lambda)ds +  \int_{0}^{t}\int_{\mathbb{R}}
e^{i\lambda y}f(s, y) W(dyds)
\end{equation}
which has a rigorous meaning. Let us denote by $ \eta_{\lambda} (t)
$ the stochastic integral in the equality above. It is known that $
\eta_{\lambda} (t) $ is a martingale \cite{W1}. The Equation
(\ref{Eq.u.Hat.Integ.Form.Appro.Special}) is equivalent to the
following linear stochastic differential equation perturbed by the
martingale $ \{\eta_{\lambda} (t), t\geq 0\} $
\begin{equation}\label{Eq.u.Hat.diff.Form.Appro.Special}
d\hat{u}_{\lambda}(t)=
({}_{\delta}\psi_{\alpha}(\lambda)+\sum_{0}^{m}c_{k}(-i\lambda)^{k})\hat{u}_{\lambda}(t)dt
+ d\eta_{\lambda} (t),
\end{equation}
with initial condition $\hat{u}_{\lambda}(t) =
\hat{u^{0}}(\lambda)$. This equation admits the explicit solution
\begin{equation}\label{Form.SoluFour}
\hat{u}_{\lambda}(t)=\hat{u}_{0}(\lambda)e^{({}_{\delta}\psi_{\alpha}(\lambda)+
\sum_{0}^{m}c_{k}(-i\lambda)^{k})t}+
\int_{0}^{t}e^{({}_{\delta}\psi_{\alpha}(\lambda)+
\sum_{0}^{m}c_{k}(-i\lambda)^{k})(t-s)}d\eta_{\lambda}(s),
\end{equation}
which is a generalization of an Ornstein-Uhlenbeck process.

\begin{prop}
Consider the equation (\ref{Eq.Princ.Four}), with $c_k=0, \forall
k\in \overline{0m}$ and $u^0\in L^p(\Omega, \mathcal{F}_0, L^2)$.
The process whose Fourier transform is given by
(\ref{Form.SoluFour}) is equivalent to the solution in the sense of
Definitions \ref{Def.mild.Solu.}, \ref{Def.Weak.Sol.1} and
\ref{Def.Weak.Sol.2}.
\end{prop}
\begin{proof}
Let $\{u(t), t\geq 0\}$ be the process whose Fourier transform is
given by (\ref{Form.SoluFour}). It is sufficient to prove that it is
the mild solution. Then  $\hat{u}(t, \lambda) $ is given for all $
\lambda $ by the following formulae

\begin{equation}\label{Form.SoluFour.Prop}
\hat{u}(t,
\lambda)=\hat{u}^{0}(\lambda)e^{{}_{\delta}\psi_{\alpha}(\lambda)
t}+
\int_{0}^{t}e^{{}_{\delta}\psi_{\alpha}(\lambda)(t-s)}d\eta_{\lambda}(s).
\end{equation}
Applying the inverse Fourier transform on the two sides of this
equation, replacing $ \eta_\lambda (s)$ by its values and using the
Fubini's Theorem, we get
\begin{eqnarray*}
u(t, .)&=& u^0* {}_{\delta}G_{\alpha}(t, .)+
\int_{0}^{t}\int_{\mathbb{R}}
\Big(\int_{\mathbb{R}}e^{-i\lambda(.-y)}
e^{{}_{\delta}\psi_{\alpha}(\lambda)(t-s)} d\lambda\Big)  f(s, y) W(dyds)\nonumber\\
&=& \int_{\mathbb{R}}{}_{\delta}G_{\alpha}(t, .-y)u^0(y)dy+
\int_{0}^{t}\int_{\mathbb{R}}
{}_{\delta}G_{\alpha}(t-s, .-y) f(s, y)W(dyds).\nonumber\\
\end{eqnarray*}
By the same calculus and using the Fourier transform, we prove that
the Fourier transform of the mild solution is given by
(\ref{Form.SoluFour.Prop}).
\end{proof}


\end{document}